\documentclass[12pt,reqno]{amsart}

\usepackage{graphicx}
\usepackage{amsmath}
\usepackage{amssymb}
\usepackage{amscd}
 \usepackage{tikz}
\usepackage{pdfpages}
\usepackage{bbm}
\usepackage{amsthm}
\usepackage{amsfonts}
\usepackage[utf8]{inputenc}
\usepackage[T1]{fontenc}
\usepackage{lmodern}
\usepackage[parfill]{parskip}
\usepackage{paralist}
\usepackage{cite}

\renewenvironment{itemize}[1]{\begin{compactitem}#1}{\end{compactitem}}
\renewenvironment{enumerate}[1]{\begin{compactenum}#1}{\end{compactenum}}

\calclayout

\newtheorem{theorem}{Theorem}
\newtheorem{theorema}{Theorem}
\newtheorem{theoremb}{Theorem}
\newtheorem{theoremc}{Theorem}
\newtheorem{theoremd}{Theorem}

\newtheorem{rk}[theorema]{Remark}
\newtheorem{lem}[theoremb]{Lemma}
\newtheorem{prop}[theoremc]{Proposition}
\newtheorem{cor}[theoremd]{Corollary}

\newcommand\bib[1]{\bibitem[#1]{#1}}
\newcommand{\su}{\mathfrak{su}}
\newcommand{\comm}[1]{}
\newcommand\1{{\bf 1}}
\newcommand\ad{{\mathfrak{ad}}}
\newcommand\End{\op{End}}
\newcommand\Hom{\op{Hom}}
\renewcommand\a{\alpha}
\renewcommand\b{\beta}

\newcommand\C{{\mathbb C}}

\newcommand\e{\varepsilon}

\newcommand\g{\mathfrak{g}}

\newcommand\h{\mathfrak{h}}

\renewcommand\l{\lambda}

\newcommand\La{\Lambda}
\newcommand\m{\mathfrak{m}}

\renewcommand\to{\mathfrak{o}}
\newcommand\op[1]{\mathop{\rm #1}\nolimits}

\newcommand\p{\partial}

\renewcommand\p{\mathfrak{p}}
\newcommand\R{{\mathbb R}}
\renewcommand\r{\mathfrak{r}}

\renewcommand\sl{\mathfrak{sl}}
\newcommand\so{\mathfrak{so}}
\renewcommand\sp{\mathfrak{sp}}
\newcommand\sym{\mathfrak{sym}}

\newcommand\vp{\varphi}
\newcommand\we{\wedge}

\newcommand\li{{\mathfrak{l}}}

\begin{document}

 \title[Non-Degenerate Almost Complex Structures in 6D]{Almost Complex Structures in 6D with Non-degenerate Nijenhuis Tensors\\ and Large Symmetry Groups}
 \author{B.\,S. Kruglikov, H. Winther}
 \address{Department of Mathematics and Statistics, NT-faculty, University of Troms\o, Troms\o\ 90-37, Norway}
 \email{ boris.kruglikov@uit.no, \quad henrik.winther@uit.no.}
 \maketitle

 \begin{abstract}
For an almost complex structure $J$ in dimension $6$ with non-degene\-ra\-te Nijenhuis tensor $N_J$, the automorphism group $G=\op{Aut}(J)$ of maximal dimension is the exceptional Lie group $G_2$. In this paper we establish that the sub-maximal dimension of automorphism groups of almost complex structures with non-degene\-ra\-te $N_J$, i.e.\ the largest realizable dimension that is less than $14$, is $\dim G=10$. Next we prove that only 3 spaces realize this, and all of them are strictly nearly (pseudo-) Kähler and globally homogeneous. Moreover, we show that all examples with $\dim \op{Aut}(J)=9$ have semi-simple isotropy.
 \end{abstract}



 \section{Introduction and main results}

Consider an almost complex manifold $(M,J)$, $J^2=-\1$, of real dimension $6$ (complex dimension $3$). The Nijenhuis tensor $N_J$ is \textit{non-degenerate} when $N_J:\La^2_\C TM\rightarrow TM$ is a ($\C$-antilinear) isomorphism of real vector spaces. For brevity, we will call an almost complex structure $J$ non-degenerate or NDG when $J$ gives rise to a non-degenerate $N_J$. Some important examples of non-degenerate almost complex structures are the critical points of the Hitchin-type functionals\cite{Br,V}, and strictly nearly Kähler (SNK) structures\cite{Na}. In this paper we also consider the indefinite analog, strictly nearly pseudo-Kähler (SNPK) structures, which are Hermitian triples $(g,\omega,J)$ on $M$ with $g$ of indefinite signature, that satisfy the same condition as in the case of definite signature:
$$
\nabla^g \omega \in \Omega^3M.
$$

The non-degeneracy of an almost complex structure guarantees that the automorphism group $\op{Aut}(J)$ is a Lie group, in particular it is finite dimensional \cite{K$_1$} and at most of dimension 14\cite{K$_2$}. Moreover, this 14 is only achieved when either $G=G_2^c\subset SO(7)$, the compact form of the exceptional complex group $G_2$ and $M=\mathbb{S}^6$ with the Calabi almost complex structure $J$, or $G=G_2^*\subset SO(3,4)$, 
the split real form of the same acting on $\mathbb{S}^{2,4}$ (see \cite{Gr,Ka} for a description of 
the homogeneous structures). These two are the maximally symmetric non-degenerate almost complex structures. The \textit{sub-maximal structures} are then the maximally symmetric among those that are not $G_2$-invariant. The purpose of this paper is to determine the structures with sub-maximal symmetry.

In addition to the automorphism group $\op{Aut}(J)$, we also consider the infinitesimal symmetry algebra $\sym(J)$. Notice that $\dim \sym(J) \ge  \dim \op{Aut}(J)$.
 \begin{theorem}
Assume $J$ is not (locally) $G_2$-symmetric. Then $\dim\sym(J)\le 10$. In the case of equality, the regular orbits of the symmetry algebra $\sym(J)$ are open (local transitivity) and $J$ is equivalent near regular points to an invariant structure on one of the homogeneous spaces
 \begin{itemize}
\item $Sp(2)/SU(2)U(1)$, which is SNK;
\item $Sp(1,1)/SU(2)U(1)$, which is SNPK of signature (4,2);
\item $Sp(4,\R)/SU(1,1)U(1)$, which is SNPK of signature (4,2).
 \end{itemize}
 \end{theorem}
\begin{cor}
The gap between maximal and sub-maximal symmetry dimensions of $\sym(J)$ for $\dim M =6$ is the same for non-degenerate almost complex structures as for SNK and SNPK.
\end{cor}
\begin{rk}
The topological types of the three homogeneous models from Theorem 1 are respectively $\C P^3$, $\C P^1 \times \C^2$ and a $\C$-line bundle over $\C P^1 \times \C$.
\end{rk}
We also investigate the possibility of singular orbits of the submaximal symmetry groups, with the conclusion that there are none. For simplicity we formulate the global version.
\begin{theorem}
Let $(M,J)$ be a connected non-degenerate almost complex manifold with $\dim \op{Aut}(J)=10$. Then $M$ is equal to the regular orbit of its automorphism group, and hence it is a global homogeneous space of one of three types indicated in Theorem 1.
\end{theorem}
We also consider non-degenerate almost complex structures with sub-submaximal symmetry:\footnote[1]{The statement about the Lie algebra of symmetries is stronger than that about the Lie group, and so we give only the local version.}
\begin{theorem}\label{Thm1}
Non-degenerate almost complex structures with $\dim \sym(J)=9$ are all (locally) homogeneous spaces $M=G/H$ with semi-simple stabilizer either $SU(2)$ or $SU(1,1)$.
\end{theorem}
The classification of 6D homogeneous almost complex structures with semi-simple isotropy was performed in \cite{AKW}. That paper also contains the classi\-fication of almost Hermitian manifolds with the same isotropy, and some speci\-fications. For instance, homogeneous SNPK structures with 9-dimensional\linebreak
  symmetries and semi-simple isotropy consist of 3 classes (in contrast with 2 in the case of SNK). 
Thus we conclude the classification of homogeneous SNPK structures with automorphism group of dimension at least 9.

We also conclude from the above theorems and \cite[Theorem 8]{AKW} 
the following statement (see also Remark 5 in loc.cit.\ about the quotients).

 \begin{cor}
The only compact homogeneous almost complex 6-dimensional manifolds with non-degenerate Nijenhuis tensor
and symmetry dimension at least 9 are $\mathbb{S}^6$, $\mathbb{C}P^3$, $\mathbb{S}^1\times\mathbb{S}^5$, 
$\mathbb{S}^3\times\mathbb{S}^3$ and their finite quotients (for symmetry dimension at least 10, 
we get only $\mathbb{S}^6$, $\mathbb{C}P^3$ but without the homogeneity assumption).
 \end{cor}

The rest of this paper constitutes a proof of the above theorems. Some computations in Maple are available as a supplement to this paper.

\textsc{Acknowledgements}: Henrik Winther is grateful to Ilka Agricola for her\linebreak 
 hospitality during his DGF-funded research stay at the University of Marburg. 
Both authors were partially supported by the Norwegian Research Council and DAAD project of Germany.

\section{Possible Isotropy Algebras}
 \begin{prop}\label{isotropyisunitary}
The isotropy algebra of the symmetries of NDG almost complex structures in 6D that has dimension $\ge3$ is either one of the
special (pseudo-) unitary algebras $\su(3)$, $\su(1,2)$, or a subalgebra of these.
 \end{prop}

Recall \cite{K$_2$} that with a Nijenhuis tensor we associate a bilinear (1,1)-form
 $$
h(v,w)=\op{Tr}[N_J(v,N_J(w,\cdot))+N_J(w,N_J(v,\cdot))]
 $$
and a holomorphic 3-form ($\op{alt}$ is the total skew-symmetrizer)
 $$
\zeta(u,v,w)=\op{alt}[h(N_J(u,v),w)-i\,h(N_J(u,v),Jw)].
 $$
When both are non-degenerate the symmetry of $(J,N_J)$ has to preserve the (pseudo-)Hermitian metric
and the holomorphic volume form, whence it is a subgroup of the special unitary group (of proper signature).

The proof of Proposition \ref{isotropyisunitary} follows the algebraic classification of NDG types of the Nijenhuis tensors \cite{K$_1$},\cite{K$_2$}:
\begin{enumerate}
\item $N(X_1,X_2)=X_2, N(X_1,X_3)=\lambda X_3, N(X_2,X_3)=e^{i\phi}X_1$
\item $N(X_1,X_2)=X_2, N(X_1,X_3)=X_2+X_3, N(X_2,X_3)=e^{i\phi}X_1$
\item $N(X_1,X_2)=e^{-i\psi}X_3, N(X_1,X_3)=-e^{i\psi} X_2, N(X_2,X_3)=e^{i\phi}X_1$
\item $N(X_1,X_2)=X_1, N(X_1,X_3)= X_2, N(X_2,X_3)=X_2+X_3$
\end{enumerate}
Here $\phi , \psi , \lambda\in\R$.
\medskip

{\bf NDG(1)}. For this class the form $h$ is non-degenerate with the exception of parameters
$\l=1$, $\vp=0,\pi$ and $\l=-1$, $\vp=\pm\pi/2$. The signature of the metric is (4,2),
the form $\zeta$ is a holomorphic volume form (for all parameters), so the
isotropy of the non-exceptional case is a subalgebra of $\mathfrak{su}(2,1)$; in the case of equality the structure
is $G_2^*$-symmetric.

For exceptional parameters note that all of them are equivalent (by a change of the complex basis $\{X_i\}_{i=1}^3$) to the case $\l=1$, $\vp=0$, i.e.
 $$
N_J(X_1,X_2)=X_2,\ N_J(X_1,X_3)=X_3,\ N_J(X_2,X_3)=X_1.
 $$
The kernel of $h$ is the complex 2-plane $\langle X_2,X_3\rangle_\C$
(it is also distinguished by the property $X\in\op{Im}(N_J(X,\cdot))$), and consequently also
the complex line $\langle X_1\rangle_\C=\C\cdot N_J(X_2\we X_3)$ is distinguished.
Thus the symmetry of the pair $(J,N_J)$ is block-diagonal, and it is easy to compute
to be equal to (as the space of $3\times3$ complex matrices)
 $$
H_0=\left\{\begin{pmatrix}e^{2i\theta} & 0\\ 0 & A\end{pmatrix}\,:\ \theta\in\R\op{mod}\pi,\
e^{i\theta}A\in\op{SL}(2,\R)\right\}.
 $$
Indeed, we write the general form
$\Phi(X_1)=e^{2i\theta}X_1$, $\Phi(X_2)=e^{-i\theta}(aX_2+bX_3)$, $\Phi(X_3)=e^{-i\theta}(cX_2+dX_3)$
of $\Phi\in\op{GL}(3,\C)$, and substitute to the defining relations $\Phi\circ J=J\circ\Phi$,
$\Phi\circ N_J=N_J\circ\La^2\Phi$, to find $a,b,c,d,\theta\in\R$, with $ad-bc=1$.

Thus the isotropy $\h_0=\mathfrak{u}(1,1)$ acts on $\C^3=\m$ with (complex) irreducible decomposition $\m=\C\oplus V$. This is the block embedding of $\h_0=\mathfrak{u}(1,1)$ into $\su(1,2)$.
Its subalgebras $\h$ of dimension 3 are $\su(1,1)$ and $\mathfrak{u}(1)\oplus\mathfrak{b}_2$, where the latter summand
is the Borel subalgebra.

\bigskip

{\bf NDG(2)}. For this class $h$ is non-degenerate with the exception of parameters $\vp=0,\pi$.
The signature of the metric is (4,2), the form $\zeta$ is is a holomorphic volume form for all parameters,
so the isotropy of the non-exceptional case is a subalgebra of $\mathfrak{su}(2,1)$.

The Nijenhuis tensor with the exceptional parameters is ($\e=\pm1$)
 $$
N_J(X_1,X_2)=X_2,\ N_J(X_1,X_3)=X_3+X_2,\ N_J(X_2,X_3)=\e X_1.
 $$
For a degenerate Hermitian structure we have:\\ $\op{Ker}(h)=\langle X_2\rangle_\C$,
$\op{Im}N_J(X_2,\cdot)=\langle X_1,X_2\rangle_\C$, $\langle X_1,X_2\rangle_\C^{\perp_h}=\langle X_2,X_3\rangle_\C$,
and finally
$\C\cdot N_J(X_2\we X_3)=\langle X_1\rangle_\C$.
Thus the symmetry of $(J,N_J)$ is given by block-diagonal (in complex coordinates) matrix with blocks of size
$1\times1$ and $2\times2$, the latter being upper-triangular. Now it is easy to compute that this
group is precisely
 $$
H_0=\left\{\begin{pmatrix}e^{2i\theta} & 0 & 0\\ 0 & \e e^{-i\theta} & \beta e^{i\theta}\\
0 & 0 & \e e^{-i\theta}\end{pmatrix}\,:\ \theta\in\R\op{mod}\pi,\
\e=\pm1,\ \beta\in\R\right\}.
 $$
Thus for this type the isotropy is at most 2D, and so should not be considered for the sub-maximal (or sub-sub maximal) problem.

\bigskip

{\bf NDG(3)}.\footnote{This form differs in \cite{K$_1$} and \cite{K$_2$}, and we use expression from the latter (the former form, though it looks differently, is equivalent).}
For this class $h$ is non-degenerate with the exception of parameters $\psi=\pm\tfrac14\pi,\pm\tfrac34\pi$;
$\vp+\psi=\pm\tfrac12\pi$; $\vp-\psi=\pm\tfrac12\pi$. Let us call these exceptional parameters of the first kind.
The signature of the metric can be both (6,0) and (4,2) (as well as the opposite (0,6), (2,4), but we do not distinguish).

The form $\zeta$ is a holomorphic volume form with the exception of parameters
$(\psi,\vp)\in\{(\pm\frac\pi6,\pm\frac\pi2),(\pm\frac\pi3,0),(\pm\frac\pi3,\pi)\}$
(here we use freedom of change of coordinates $X_2\leftrightarrow X_3$ resulting in identification
$(\psi,\vp)\sim(\psi+\pi,\vp+\pi)$). Call these exceptional parameters of the second kind.

Therefore the isotropy of the non-exceptional case is a subalgebra of $\mathfrak{su}(3)$ or
$\mathfrak{su}(2,1)$; the case of equality corresponds to $G_2$ or $G_2^*$-symmetric structures~$J$.

Consider at first exceptional parameters of the first kind. The
Nijenhuis tensor with the exceptional parameters is obtained by substitution of the above values to
 $$
N_J(X_1,X_2)=e^{-i\psi}X_3,\ N_J(X_1,X_3)=-e^{i\psi}X_2,\ N_J(X_2,X_3)=\e X_1.
 $$
Change of basis $X_2\leftrightarrow X_3$ results in $(\psi,\vp)\mapsto(\psi+\pi,\vp+\pi)$, so we can fix
$\psi=\pm\tfrac14\pi$ in the first case, but cannot modify the conditions on exceptional parameters any more.

For the exceptional parameters (as listed) we have: $\op{Ker}(h)$ is $\langle X_1\rangle_\C$,
$\langle X_2\rangle_\C$ or $\langle X_3\rangle_\C$ respectively in the generic exceptional case
or $\langle X_1,X_2\rangle_\C$, $\langle X_1,X_3\rangle_\C$ and $\langle X_2,X_3\rangle_\C$ in the
case of strong degeneration (i.e.\ intersection of two conditions).

For any of the 3+3 exceptional cases we get:
$\C\cdot N_J(X_1\we X_2)=\langle X_3\rangle_\C$,\linebreak 
$\C\cdot N_J(X_1\we X_3)=\langle X_2\rangle_\C$,
$\C\cdot N_J(X_2\we X_3)=\langle X_1\rangle_\C$;
$\op{Im}N_J(X_1,\cdot)=\langle X_2,X_3\rangle_\C$, $\op{Im}N_J(X_2,\cdot)=\langle X_1,X_3\rangle_\C$,
$\op{Im}N_J(X_3,\cdot)=\langle X_1,X_2\rangle_\C$.
Thus in any case the symmetry acts by block-diagonal matrix with $1\times1$ and $2\times2$ blocks
(where $\op{Ker}(h)$ is $1\times1$ block in the first 3 cases, and $2\times2$ block in the second 3 cases).

Consider at first generic exceptional cases. If $\op{Ker}(h)=\langle X_1\rangle_\C$, then the complex-linear operator
$Z\mapsto N_J(X_1,N_J(X_1,Z))$ has eigenvalues $i,-i$ on $\langle X_2,X_3\rangle_\C$,
and this distinguishes $\langle X_2\rangle_\C$ and
$\langle X_3\rangle_\C$, so the symmetry is a subgroup of the diagonal 
$U(1)\times U(1)\times U(1)\subset\op{GL}(3,\C)$,
and we compute it to be
 $$
H_0=\{\op{diag}(e^{i\alpha},e^{i\beta},e^{i\gamma}):\alpha+\beta+\gamma=0\op{mod}2\pi\}.
 $$
This is 2D, so is discarded. The same happens if $\op{Ker}(h)=\langle X_2\rangle_\C$ with
the complex-linear operator $Z\mapsto N_J(X_2,N_J(X_2,Z))$ on $\langle X_1,X_3\rangle_\C$,
and also if\linebreak
 $\op{Ker}(h)=\langle X_3\rangle_\C$ with
the complex-linear operator $Z\mapsto N_J(X_3,N_J(X_3,Z))$ on $\langle X_1,X_2\rangle_\C$.
Thus generic exceptional cases do not carry a sub-maximally symmetric NDG almost complex structure.

Consider now the cases of strong degeneration. If $\op{Ker}(h)=\langle X_1,X_2\rangle_\C$
($\psi=\pm\tfrac14\pi$, $\vp=\pm\tfrac12\pi-\psi$), then the complex-linear operator
$Z\mapsto N_J(X_1,N_J(X_1,Z))$ has eigenvalues $i$ (double) or $-i$ (double), so we cannot reduce to
the diagonal case. So consider the general block from.

Our parameters are $\psi=\e_1\frac\pi4$, $\vp=\e_2\frac\pi2-\e_1\frac\pi4$, and for them
 \begin{gather*}
N_J(X_1,X_2)=\tfrac1{\sqrt{2}}(X_3-\e_1JX_3),\
N_J(X_1,X_3)=-\tfrac1{\sqrt{2}}(X_2+\e_1JX_2),\\
N_J(X_2,X_3)=\tfrac1{\sqrt{2}}(\e_1\e_2X_1+\e_2JX_1);
 \end{gather*}
where $\e_1,\e_2$ are equal to $\pm1$. Now we write the general form
$\Phi(X_1)=e^{-i\theta}(aX_1+bX_2)$, $\Phi(X_2)=e^{-i\theta}(cX_1+dX_2)$, $\Phi(X_3)=e^{2i\theta}X_3$,
and substitute to the defining relations $\Phi\circ J=J\circ\Phi$,
$\Phi\circ N_J=N_J\circ\La^2\Phi$. Then we get either $\e_1\e_2=1$ and
 $$
H_0=\left\{\begin{pmatrix}e^{i(\a-\theta)}\cos\nu & e^{i(\b-\theta)}\sin\nu & 0\\
-e^{-i(\b+\theta)}\sin\nu & e^{-i(\a+\theta)}\cos\nu & 0\\
0 & 0 & e^{2i\theta}\end{pmatrix}\,:\ \a,\b,\theta,\nu\in\R\right\},
 $$
or $\e_1\e_2=-1$ and
 $$
H_0=\left\{\begin{pmatrix}e^{i(\a-\theta)}\cosh\nu & e^{i(\b-\theta)}\sinh\nu & 0\\
e^{-i(\b+\theta)}\sinh\nu & e^{-i(\a+\theta)}\cosh\nu & 0\\
0 & 0 & e^{2i\theta}\end{pmatrix}\,:\ \a,\b,\theta,\nu\in\R\right\}.
 $$
These groups equal $SU(2)\times U(1)$ and $SU(1,1)\times U(1)$, and so $\h_0=\mathfrak{u}(2)$ or $\h_0=\mathfrak{u}(1,1)$, respectively.
The cases of $\op{Ker}(h)=\langle X_1,X_3\rangle_\C$ and $\op{Ker}(h)=\langle X_2,X_3\rangle_\C$ are treated similarly, and result in the same symmetry groups.

Now let us consider exceptional parameters of the second kind. Recall from \cite{K$_1$} that
we have two anti-holomorphic maps
$\Phi_1:\C P^2\rightarrow\op{Gr}_2^\C(3)\simeq\C P^2$, $\C\langle X\rangle\mapsto\op{Im}N(X,\cdot)$, and
$\Phi_2:\op{Gr}_2^\C(3)\rightarrow\C P^2$, $\C^2\langle Y,Z\rangle\mapsto\C\langle N(Y,Z)\rangle$.
For non-degenerate $N_J$ the composition $\Phi=\Phi_2\circ\Phi_1$ is a bi-holomorphism of $\C P^2$.

Direct computation shows that it has precisely 3 fixed points $\langle X_1\rangle_\C$,
$\langle X_2\rangle_\C$, $\langle X_3\rangle_\C$, provided $\psi\not\equiv\pm\vp\op{mod}\pi$ and
$2\psi\not\equiv0\op{mod}\pi$. Our exceptional parameters of the second kind satisfy these inequalities,
so the symmetry is a subgroup of the diagonal $U(1)\times U(1)\times U(1)\subset\op{GL}(3,\C)$,
and we compute it to be
 $$
H_0=\{\op{diag}(e^{i\alpha},e^{i\beta},e^{i\gamma}):\alpha+\beta+\gamma=0\op{mod}2\pi\}.
 $$
This is 2D, so is discarded.

\bigskip

{\bf NDG(4)}. For this class $h$ is non-degenerate (without exceptions) with the
signature (4,2). The form $\zeta$ is a holomorphic volume form for all parameters.
Hence the isotropy is a subalgebra of $\mathfrak{su}(2,1)$.

\section{The case of locally transitive $\op{Aut}(J)$}
 By \cite{K$_2$} we know that if $G$ preserves a non-degenerate $J$, then the isotropy representation is faithful on the isotropy algebra $\h\subset \g$. By Proposition \ref{isotropyisunitary},  $\h$ is a subalgebra of either $\su(3)$ or $\su(1,2)$ in their standard representations. We can assume that $\h$ is a proper subalgebra, because otherwise $G$ has dimension 14 and hence is maximal.

 From \cite{Bu} we know that there is an invariant SNK structure on $Sp(2)/U(2)$, and SNK always has NDG $N_J$. This is a non-degenerate structure $J$ with 10D symmetry. To achieve symmetry of dimension $\ge 10$ as in Theorem 1, we have to consider the case when $\dim \h\ge\dim \mathfrak{u}(2)=4$. Thus, up to conjugacy the possible $\h$ for sub-maximal almost complex structures are:
 \begin{itemize}
 \item $\mathfrak{u}(2)\subset \su(3)$.
 \item $\mathfrak{u}(2)\subset \su(1,2)$.
 \item $\mathfrak{u}(1,1)\subset \su(1,2)$.
 \item The (only) parabolic subalgebra $\p\subset \su(1,2)$ of dimension 5.
 \item A 4D maximal subalgebra $\r$ of $\p$.
 \end{itemize}
  In addition to this, for the needs of Theorem 2 we shall consider the subalgebras $\h\subset\p$ with $\dim \h \ge 3$. These subalgebras are most easily described by considering the $\mathbb{Z}$-grading which exists on a parabolic subalgebra.
 \begin{prop}
 The algebra $\p=\oplus_{i\in\mathbb{Z}} \p_i$ is solvable and graded, with $\dim(\p_0)=2$, $\dim(\p_1)=2$, $\dim(\p_2)=1$ and $\dim(\p_i)=0$ for all other $i$. As a Lie algebra it is the extension of the 3D Heisenberg algebra $\mathfrak{heis}_3=\C\oplus\R$ by derivations $\mathfrak{gl}(1,\C)\subset\mathfrak{gl}_2(\R)\subset\mathfrak{Der}(\mathfrak{heis}_3)$.
 \end{prop}
 The proper subalgebras $\h$ of $\p$ with $\dim \h\ge3$ are then given by specifying $\dim \h\in\{3,4\}$ and $\dim(\h\cap \p_0)$. It can be shown that any such subalgebra is conjugate to a graded subalgebra in $\p$. This gives 4 non-equivalent possibilities for $\h\subset \p$:
 \begin{itemize}
 \item The 4D subalgebras $\r$, which have $\dim(\r\cap \p_0)=1$, a 1D family specified by which 1D subalgebra of $\p_0$ is included.
 \item The 3D nilradical $\li_0=\p_1\oplus \p_2$ of $\p$, which has $\dim(\li_0\cap \p_0)=0$.
 \item The 3D subalgebra $\li_1$ which has $\dim(\li_1\cap \p_0)=1$, and contains the grading element $s$ of $\p$: $\li_1=\R s \oplus \R v \oplus \p_2$, the choice of $v\in\p_1$ is not essential.
 \item The 3D subalgebra $\li_2$ which has $\dim(\li_2\cap \p_0)=2$, and contains no elements from $\p_1$, i.e. $\li_2=\p_0\oplus \p_2$.
 \end{itemize}
 The list of subalgebras of $\su(1,2)$ is the same as the one found in \cite{PWZ}, but their naming convention is different. All algebras given here come equipped with a representation $\m$, the restriction of the tautological representations of $\su(3)$ or $\su(1,2)$. These algebras are not all reductive, so the representation of $\h$ on $\g$ may not split into a direct sum of $\m$ and $\h$. This means that the homogeneous space may not be reductive. The quotient $\h$-module $\g / \h$ must however be of the given type $\m$.

 \subsection{The $\h$-module structure of $\g$}
 In the event that $\g$ does not split into a direct sum of $\h$ and $\m$, we choose an arbitrary complement of $\h$ which we will still denote by $\m$, even though it is not a submodule. We have
 \begin{align*}
 [h,m]=\mu(m)h+h\,m\in \h \oplus \m
 \end{align*}
 for some $\mu:\m\rightarrow \End(\h)$. Here $h\,m$ denotes the action of $\h$ on the module $\m=\g/\h$. Let us change the complement $\m$ by some operator $A:\m\rightarrow \h$, so that the new complement is $\m_{new} = \{(A(m),m)|m\in\m \}$. Then
 \begin{align*}
 [h,m+Am]=\mu(m)h+[h,Am]-A(h\,m)+h\,m+A(h\,m)\in \h \oplus \m
 \end{align*}
 and the first three terms describe $\mu_{new}$. If we denote by $d_\h$ the Lie algebra differential in the complex $\Lambda^\ast \h^\ast \otimes \m^\ast \otimes \h$ of $\Hom(\m,\h)$-valued forms on $\h$, this can be written as
 \begin{align*}
 \mu_{new}=\mu+d_\h A.
 \end{align*}
 Moreover, from the Jacobi identity between elements $m,h_1,h_2$ we get $d_\h \mu =0$, so $\mu$ is a cocycle. This gives the following statement (it can also be seen as a result of the isomorphism $\op{Ext}^1_\mathfrak{h}(\mathfrak{m},\mathfrak{h})=H^1(\mathfrak{h},\op{Hom}(\mathfrak{m},\mathfrak{h}))$ and the extension obstruction for modules \cite{Gi}).
 \begin{lem}\label{cohomologylemma}
 The equivalence classes of $\h$-modules $\g$ with $\g /\h \simeq \m$ are given by the Lie algebra cohomology $H^1(\h,\Hom(\m,\h))$. In particular, if the cohomology vanishes, then $\g=\h\oplus\m$ is a direct sum.
 \end{lem}
 The computation of this cohomology was performed in Maple, and worksheets are available in the supplement. The result is the following.
 \begin{prop}\label{cohomologylist}
 For the reductive subalgebras $\h$ of $\su(3)$, $\su(1,2)$ with $\dim\h\ge3$, we have $\dim H^1(\h,\Hom(\m,\h))=0$. Let $s\in\p$ be the grading element. For the solvable subalgebras $\h$ of $\su(3)$, $\su(1,2)$ with $\dim\h\ge3$, we have
 \begin{itemize}
 \item $\dim H^1(\p,\Hom(\m,\p))=0$,
 \item $\dim H^1(\r,\Hom(\m,\r))=6$ when $s\in\r$,
 \item $\dim H^1(\r,\Hom(\m,\r))=0$ when $s\not\in\r$,
 \item $\dim H^1(\li_0,\Hom(\m,\li_0))=10$,
 \item $\dim H^1(\li_1,\Hom(\m,\li_1))=4$,
 \item $\dim H^1(\li_2,\Hom(\m,\li_2))=0$.
 \end{itemize}
 \end{prop}
 Hence the $\h$-module $\g$ decomposes into a direct sum $\g=\m\oplus\h$ when $\h$ is reductive, $\h=\p$, $\h=\li_2$ or $\h=\r$ for $s\not\in\r$. In the other cases, there are non-decomposable $\h$-modules $\g$ which satisfy $\g/\h=\m$, and these are parameterized by elements of the cohomology.
\subsection{Lie algebra structures on the $\h$-module $\g$}
Let $\h$ be a Lie algebra and $\g$ be an $\h$-module such that $\h\subset\g$ as a submodule. By a \textit{Lie algebra extension}\footnote[1]{This is different from "right" or extensions by derivations\cite{F}.} of $\h$ on $\g$, we mean a bracket operation
\begin{align*}
[,]:\La^2\g\rightarrow\g
\end{align*}
which satisfies the usual Lie algebra axioms and the restriction criteria that
\begin{align*}
&[,]:\La^2\h\rightarrow\h\\
&[,]:\h\wedge \g\rightarrow\g
\end{align*}
are respectively the Lie bracket of $\h$ and the module action of $\h$ on $\g$.

\begin{lem}
Those Jacobi identities of the bracket $[,]$ which involve an element from $\h$ are equivalent to the $\h$-equivariancy of $[,]$.
\end{lem}
\begin{proof}
Let $\m$ be a complement to $\h$ in $\g$ as a vector space. The Jacobi relation involving 3 elements from $\h$, $\op{Jac}_\h:\La^3\h\rightarrow\h$, vanishes as $\h$ is a Lie algebra. The Jacobi relation involving 2 elements from $\h$ and 1 from $\m$ vanishes
as $\m$ is an $\h$-module. Finally the Jacobi relation involving 1 element from $\h$ and 2 from $\m$
is precisely the equivariancy of the map $[,]$.
\end{proof}
\begin{cor}\label{cor2}
The bracket $[,]\in(\La^2\g^\ast\otimes\g)^\h$ which satisfies the restriction criteria is a Lie algebra extension iff the Jacobi identity $\op{Jac}_\m$ vanishes on a complement $\m$ of $\h$ in $\g$.
\end{cor}
We denote the space of elements of $\g^\ast$ which vanish on $\h$ by $\m^\ast$. This can be identified with the dual space of a complement $\m$ to $\h$, and $\m^\ast$ is a submodule of $\g^\ast$. Thus
\begin{align*}
(\La^2\m^\ast\otimes\g)^\h\subset(\La^2\g^\ast\otimes\g)^\h
\end{align*}
is also a submodule. We call $B(\h,\g)=(\La^2\m^\ast\otimes\g)^\h$ the space of $\h$-equivariant brackets. We have that if $\theta\in (\La^2\g^\ast\otimes\g)^\h$ satisfies the restriction criterion and $\phi\in B(\h,\g)$, then $\theta+\phi$ also satisfies the criterion. Thus the Lie algebra extensions are contained in an affine subspace of $(\La^2\g^\ast\otimes\g)^\h$ which is modeled on $B(\h,\g)$. By Corollary \ref{cor2}, the bracket $\phi+\theta$ defines a Lie algebra extension iff it satisfies $\op{Jac}_\m (\phi+\theta)=0$.

\subsubsection{Solvable Isotropy}
The list of possible isotropy algebras $\h$ and $\h$-modules $\g$ which preserve an almost complex structure and Nijenhuis tensor on $\g/\h=\m$ with $\dim\m=6$ is given by Lemma \ref{cohomologylemma} and Proposition \ref{cohomologylist}. The $\h$-modules $\g$ are parameterized by choosing representatives $\mu\in \h^\ast \otimes \m^\ast \otimes \h, d_\h\mu=0$ via a splitting of the canonical projection $Z^1(\h,\m^\ast \otimes \h)\rightarrow H^1(\h,\m^\ast \otimes \h)$.

Given an arbitrary complement $\m$ to $\h$ in $\g$, these representatives are maps $\mu:\h\times\m\rightarrow\h$. The representation then consists of block matrices with respect to the decomposition $\g=\h\oplus\m$, and $\mu$ describes the upper-triangular block. In particular, the representation matrices are block diagonal when the cohomology vanishes.

Consider $\alpha\in \La^2\m^\ast\otimes\g$. We have that $\alpha\in B(\h,\g)$ if $h\,\alpha=0$ for all $h\in\h$. When the cohomology does not vanish, this system consists of linear equations in the parameters of $\alpha$, and quadratic equations in the parameters of both $\mu$ and $\alpha$. Let $\alpha_0$ be an element in the solution space of the linear equations. The Jacobi identity for $\alpha_0$ is a system of equations in the parameters of $\alpha_0$ and $\mu$ which contains linear equations in the parameters of $\mu$. These imply in each case that $\mu=0$ whenever the Jacobi identity is satisfied. This yields the following proposition.

\begin{prop}
The only $\h$-modules $\g$ which admit Lie algebra extensions are the direct sums $\g=\h\oplus\m$.
\end{prop}

The module $\g$ is thus unique for each $\h$, and we may compute the space of invariant brackets, parameterize this, and solve the system of linear and quadratic equations in the parameters given by the Jacobi identity.

The solvable isotropy algebras are $\p\subset\su(1,2)$ and the subalgebras of $\p$. The sum of the center and the Borel subalgebra in $\mathfrak{u}(1,1)$ is also solvable, but this is equivalent to $\li_2\subset \p$. The algebras $\p$ and $\r$ have dimensions 5 and 4, and hence could be the isotropy of a sub-maximal model. These relate to Theorem 1. The other solvable algebras $\li_k$ have $\dim \li_k=3$, and are investigated in order to prove Theorem 3.

\begin{tabular}[t]{ | l | l | l |}
    \hline
    $\h$ & $\dim(B(\h,\g))$ &  solutions to Jacobi \\ \hline
       $\p$      & 2  & 0\\
       $\r$      & 2  & 1\\
       $\li_0$  & 24  & 18\\
       $\li_1$  & 2  & 0\\
       $\li_2$  & 6  & 1\\

    \hline
\end{tabular}\\

In this table, the third column indicates the number of families of solutions to the Jacobi equations, and they always come with some free parameters. This number is not invariant and depends on the parametrization used for the solution set, and in particular the number of families for $\h=\li_0$ can change depending on this. Each solution to the Jacobi equations corresponds to a homogeneous space (including $H$-invariant Lie groups) equipped with an invariant almost complex structure $J$. We compute (in Maple) the Nijenhuis tensor $N_J$ from the structure constants given by the solution, with the following result:
\begin{prop}
The invariant almost complex structures on the homogeneous spaces with solvable isotropy algebra have degenerate Nijenhuis tensor.
\end{prop}
This shows that there are no cases of non-degenerate homogeneous almost complex structures with solvable isotropy algebra $\h$, $\dim \h\ge3$. One half of Theorem 3 follows immediately, as any 3D Lie algebra is either solvable or semi-simple, so if $\g$ is locally transitive with $\dim \g=9$ its 3D isotropy $\h$ must be semi-simple. What remains is to rule out sufficiently big non-transitive symmetry algebras. This will be done in section 4.

\subsubsection{Non-Solvable Isotropy}
When the isotropy algebra is not solvable, then either $\h=\mathfrak{u}(2)$, $\h=\mathfrak{u}(1,1)$, or $\h$ is semi-simple. The latter possibility was considered in \cite{AKW}, so we ignore it here. Thus $\dim \h =4$, and the rest of this section is devoted to the proof of Theorem 1. The isotropy representation decomposes into submodules (see Proposition \ref{isotropyisunitary} and the beginning of Section~3):
\begin{align*}
\m=V\oplus\C.
\end{align*}
The $\C$ term is a trivial representation of $\h_{ss}=\su(2)$ or $\h_{ss}=\su(1,1)$ ($\h_{ss}$ is the semi-simple part of $\h$), but it is irreducible with respect to the center $\mathfrak{u}(1)$ of $\h$, and $V$ is equivalent to the tautological action of $\h$ on $\C^2$. By the Levi decomposition, either $\g=\h\oplus\m$ is semi-simple, or there is a solvable radical $\mathfrak{r}\subset\g$.

Let's consider the semi-simple case first. Since $\dim(\g)=10$, the algebra $\g$ is a real form of $B_2\simeq C_2$. The real forms of $B_2$ are $\mathfrak{so}(5)\simeq\sp(2)$, $\mathfrak{so}(1,4)\simeq\sp(1,1)$, and $\mathfrak{so}(2,3)\simeq\sp(4,\R)$. Since these $\g$ are pseudo-orthogonal, an embedding into them is the same as a real "defining" representation $\varphi:\h\rightarrow \op{End}(\R^5)$ which preserves a non-degenerate symmetric bilinear form $g$. The signature of $g$ then determines the algebra $\g$. We may compute the isotropy representation of $\h$ on $\g/\h$ from $\varphi$ by the $\g$-equivariant isomorphism $\La^2\R^5\simeq \mathfrak{so}(\R^5,g)=\g\subset \End(\R^5)$. The following comes as a result of simple case-by-case considerations.
\begin{prop}
The only defining representation $\varphi:\h\rightarrow\g$ which produces the correct isotropy representation $\m=\g/\h$ is the one with $\h-$module decomposition $\R\oplus V=\R^5$.
\end{prop}
Thus $g$ must be a sum of invariant forms on each submodule. For $\h=\mathfrak{u}(1,1)$, the invariant form on $V$ has signature $(2,2)$, so $g$ has signature $(2,3)\simeq (3,2)$. For $\h=\mathfrak{u}(2)$, the invariant form on $V$ has signature $(4,0)\simeq(0,4)$, so $g$ has signature $(4,1)\simeq (1,4)$ or $(5,0)\simeq (0,5)$ depending on the sign of the $\R$-component.
\begin{cor}
There are only 3 embeddings of $\h$ into the real forms of $B_2$ with the given isotropy:
\begin{itemize}
\item $\h=\mathfrak{u}(1,1)\rightarrow \mathfrak{so}(2,3)$
\item $\h=\mathfrak{u}(2)\rightarrow \mathfrak{so}(5)$
\item $\h=\mathfrak{u}(2)\rightarrow \mathfrak{so}(1,4)$.
\end{itemize}
\end{cor}
These injective Lie algebra homomorphisms integrate into injective homomorphisms of Lie groups $H\rightarrow G$. We may explicitly compute $N_J$ for the invariant $J$ in each case, with the following result:
\begin{prop}
For each of these embeddings there are (up to overall sign) two $G$-invariant almost complex structures $J$ on $G/H$. One corresponds to a \mbox{(pseudo-)} Kähler structure and has vanishing $N_J$, and the other to a SN(P)K structure and has non-degenerate $N_J$. Both have the same signature. For $\g=\mathfrak{so}(2,3)$ and $\g=\mathfrak{so}(1,4)$ the signature of the metric is (4,2), and for $\g=\mathfrak{so}(5)$ the signature is (6,0).
\end{prop}
Since we showed in the previous section that the possibility of 5D $\h=\p$ is not realized, these examples equipped with the almost complex structure $J$, which is SNPK, realize sub-maximal symmetry dimension.

Suppose now that $\m$ is semi-simple. Then $\g=\mathfrak{u}(1)\oplus\h_{ss}\oplus \m$, but $\mathfrak{u}(1)$ is not central, as we prescribed the action of $\h=\mathfrak{u}(1)\oplus\h_{ss}$, hence $\mathfrak{u}(1)$ acts as a derivation of $\h\oplus\m$. By Whitehead's lemma\cite{F} all derivations of semi-simple Lie algebras are inner derivations, that is belong to the image of the map\linebreak 
$\mathfrak{ad}:\g\rightarrow \mathfrak{Der}(\g_{ss})=\mathfrak{Der}(\h_{ss}\oplus\m)\simeq\h_{ss}\oplus\m$. For dimensional reasons this map has a non-trivial kernel. Since the kernel is a 1D $\h-$submodule of $\g$, it must be $\mathfrak{u}(1)$, thus $\mathfrak{u}(1)$ is central, but $\mathfrak{u}(1)$ acts as a non-zero derivation, and this is a contradiction.

Suppose the semi-simple Levi factor $\g_{ss}$ of $\g$ is larger than $\h_{ss}$, but smaller than $\g$. By the above, it shall not contain $\m$. A semi-simple subalgebra has dimension at least 3, which means that the radical of $\g$ is the $\h$-submodule $V$. The derived subalgebra of the radical is also an $\h$-submodule (because $\h$ are derivations of $V$). The radical is solvable, so its derived subalgebra is a proper submodule. Therefore the radical is Abelian. Hence the Nijenhuis tensor is degenerate.

Finally let's consider the case where $\g_{ss}=\h_{ss}$, so the radical of $\g$ is $\mathfrak{r}=\m\oplus\mathfrak{u}(1)$. Then the derived subalgebra of $\mathfrak{r}$ is $\m$, as it must be a proper submodule of $\mathfrak{r}$ including $\m$ (due to the action of $\mathfrak{u}(1)$). Then $\m$ is nilpotent, and the derived subalgebra of $\m$ will be either $\C$ or $V$. In either case one $\h$-submodule will not be in the image of the brackets on $\m$, and since these modules are complex, the same module is not in the image of the Nijenhuis tensor. Hence the Nijenhuis tensor is degenerate. This concludes the proof of Theorem 1.

 \section{Locally intransitive $\op{Aut}(J)$}
When the symmetry group $G$ is not locally transitive, the $G$-manifold $M$ (or its invariant open subset) is not (naturally, locally) homogeneous. Therefore the full range of algebraic tools we used in the previous section is unavailable to us. Instead, we can find a foliation by $G$-orbits in a neighbourhood of any regular point $x\in M$. The leaves must have positive codimension, and each leaf is a local homogeneous space of $G=\op{Aut}(J)$ in its own right. We may therefore investigate the existence of lower dimension homogeneous spaces $O$ whose isotropy algebra admits the existence of an invariant non-degenerate Nijenhuis tensor on the tangent space $\m$ of a regular point of $M$. This means that the full isotropy representation $\m$ must be one of those discussed in the previous section.

The tangent space $T_xO=\to$ of the orbit through $x$ must be an invariant subspace of $\m$ for the isotropy algebra $\h$. The isotropy $\h$ is still represented effectively (now on $\to$) as before, so the dimension of the symmetry algebra $\g$ is $\dim\g=\dim\to+\dim\h$. This means the possible pairs $(\h,\to)$ which have combined dimension $\dim\g\ge 9$ are the following:
\begin{itemize}
 \item $ \h= \p \subset \su(1,2)$, $\dim\to=4$.
 \item $ \h= \r \subset \su(1,2)$, $\dim\to=5$, this $\r$ is the unique possible 4D isotropy which has a 5D submodule $(s\in\r)$.
\end{itemize}
We also have the following lemma:
\begin{lem}
The quotient $\h$-module $\m/\to$ is a trivial module.
\end{lem}
\begin{proof}
The orbits locally foliate $M$. There exist local coordinates $(x,y)$ on $M$ such that the leaves (which are the flows of $\g$) have the form $\{(x,y): y^j=c^j\}$ for some constants $c^j$. In these coordinates $\g$ is generated by vector fields of the form $X=f^i(x,y)\partial_{x^i}$, and $\h$ has block form, which is equivalent to the claim.
\end{proof}
Neither of the possible choices $(\h,\m,\to)$, which satisfy $\dim\h\ge 3$, also satisfies this condition, hence these triples must be discarded. Indeed in both cases, the grading element $s\in\h$ acts non-trivially on $\m/\to$. We conclude that no non-degenerate almost complex structure $J$ with locally intransitive symmetry algebra $\g$ satisfies $\dim\g\ge 9$.

In the locally transitive case, considerations from section 3 show that $\h$ is semi-simple, so $\h=\su(2)$ or $\h=\su(1,1)$. Such $(\g,\h,\m,J)$ were classified in \cite{AKW}. This completes the proof of Theorem~3.

\section{The sub-maximal models are globally homogeneous}
In this section we prove Theorem 2. In Sections 3 and 4, we proved that the regular orbits $O_\text{reg}$ of the sub-maximal models are open in $M$, and are homogeneous spaces of $G$. Let us write $O_\text{reg}=G/H_\text{reg}$, with $G$ and $H_\text{reg}$ as found in Section 3, i.e. $G$ is one of $\text{Sp}(2),\text{Sp}(1,1)$, $\text{Sp}(4,\R)$ and $H_\text{reg}$ is a 4D subgroup. Throughout this section, $\g$ is $\sp(2)$, $\sp(1,1)$ or $\sp(4,R)$.

In addition to the regular (open) orbits described in Section 3, there could a priori be singular orbits (of positive codimension). Such orbits must also be homogeneous spaces of the symmetry group $G$.

The candidates for homogeneous singular orbits are enumerated by conjugacy classes of subalgebras $\h\subset \g$ with $\dim \h > \dim \h_\text{reg}=4$. In addition to $\g$ itself, we must consider the maximal subalgebras (and their subalgebras). By \cite{M} (see also \cite{GOV}), for a real semi-simple Lie algebra $\g$ a maximal subalgebra is parabolic, semi-simple or the stabilizer of a pseudo-torus. The list of such subalgebras with $\dim\h>4$ is
\begin{itemize}
 \item $ \h=\p_1 \subset \sp(4,\R)$, $\dim\h=7$, parabolic,
 \item $ \h=\p_2 \subset \sp(4,\R)$, $\dim\h=7$, parabolic,
 \item $ \h=\p_2 \subset \sp(1,1)$, $\dim\h=7$, parabolic,
 \item $ \h=\so(4)$, $\dim\h=6$, semi-simple,
 \item $ \h=\so(1,3)$, $\dim\h=6$, semi-simple,
 \item $ \h=\so(2,2)$, $\dim\h=6$, semi-simple.
\end{itemize}
The parabolics $\p_1$ and $\p_2$ are labelled with respect to the name for $\g$ specified (Dynkin diagram $C_2$), so $\p_1$ excludes the root space of the shorter simple root of $\sp(4,\R)$, and $\p_2$ excludes the longer. Where the embedding is not specified, there are embeddings to several different $\g$. All the pseudo-toric stabilizers have dimension $\le 4$, and that's why they are excluded from the list.

The orbit itself does not need to be almost complex, but the almost complex structure on $M$ still yields some structure on the orbit $O$. Let $\to=T_xO$ denote the tangent space of a point $x\in O$.
\begin{prop}
Suppose $O$ is a singular orbit. Then either $O$ admits a $G$-invariant complex distribution $L^2$ or $L^4$ ($J$-invariant subspaces of $\to$), or $O$ is totally real (meaning $J\to \cap \to=\{0\}$, $\to\not=0$), or $O$ is an invariant point (that is, $\to=0$).
\end{prop}
\begin{proof}
Consider the restriction of $J$ to $\to$. Since $J$ and $O$ are $G$-invariant, the intersection between the image $J\to\subset T_x M$ and $\to$ at the point $x\in O$ is $H$-invariant, where $H$ is the stabilizer of $x$. Call this space $L_x = J\to \cap \to$. The distribution $L$ given by $L_x$ for each $x\in O$ is thus $G$-invariant. Since $J^2=-1$, $J|_L$ is an almost complex structure on $L$. Hence the dimension of $L$ can be 0, 2 or 4, while 6 is not possible since $O$ is singular. If the dimension is 0 then $O$ is totally real or an invariant point.
\end{proof}
We treat each case separately.
\subsection{Invariant points}
At an invariant point $x$, $\h=\g$. Since $\g$ is a simple algebra, the isotropy representation $\g\rightarrow \End(T_xM)$ is either faithful or trivial. It cannot be faithful, because the smallest nontrivial complex module $V$ of $C_2$ has $\dim_\C V=4$.

Thus the isotropy representation is trivial. Recall the Thurston stability theorem\cite{T}, which states that if a nontrivial Lie group action has a fixed point with trivial isotropy representation, then $H^1(G,\R)\not=0$. Nonzero cohomology classes in $H^1(G,\R)$ correspond to nontrivial homomorphisms from $G$ to $\R$, and since $G$ is a simple Lie group in our case there are no such homomorphisms. Therefore $H^1(G,\R)=0$ and the sub-maximal model has no invariant points.
\subsection{Totally real orbits}
If $O$ is totally real, it can at most have dimension 3. On the other hand, the maximal dimension of a proper subalgebra $\h$ of $\g$ is 7 (achieved by maximal parabolics of $\sp(1,1)$ and $\sp(4,\R)$), while $\dim \g=10$. Therefore we have $\dim O =3$.

\begin{lem}\label{orbitrealthree}
If $O$ is a totally real orbit of dimension 3, there exists at least one nontrivial $\h$-invariant map $\La^2\to\rightarrow \to$.
\end{lem}
\begin{proof}
Since $\to$ is totally real, $\to\oplus J\to=T_xM$ and this decomposition is $H$-invariant, which yields an invariant projection $\pi:T_xM \rightarrow \to$. The Nijenhuis tensor $N_J$ is non-degenerate, so the restriction $N_J|\to:\La^2\to_x\rightarrow T_x M$ is injective. Write $L=N_J(\La^2\to)$, so $\dim L=3$. At least one of the maps $\pi:L\rightarrow\to$ and $\pi\circ J:L\rightarrow \to$ must be nonzero, call such a map $p$. Then $p\circ N_J:\La^2\to\rightarrow \to$ is a nontrivial $H$-invariant map.
\end{proof}
The 7D maximal parabolics are $\p_2\subset \sp(1,1)$, which is $|1|$-graded, $\p_2\subset \sp(4,\R)$, which is $|1|$-graded, and $\p_1\subset \sp(4,\R)$, which is $|2|$-graded. Each parabolic has a grading element, which acts on $\to$ as a real scalar when the parabolic is $|1|$-graded. Since a scalar action with weight $\lambda\not=0$ on $\to$ will have weight $2\lambda$ on $\La^2\to$, the $|1|$-graded parabolics do not admit any maps of the type we constructed in Lemma~\ref{orbitrealthree}. Hence only the $|2|$-graded $\p_1$ is interesting. In this case there is a splitting $\to=\to_1\oplus \to_2$, with $\dim\to_1=2$ and $\dim\to_2=1$, which is invariant with respect to the 0-graded piece of $\p_1$, $(\p_1)_0\simeq \mathfrak{sl}(2)\oplus \R$. Here the $\R$ term is generated by the grading element, which acts with weight $1$ on $\to_1$ and $2$ on $\to_2$. The action of $\mathfrak{sl}(2)$ on $\to_1$ is equivalent to the tautological action on $\R^2$, which admits a scalar valued invariant 2-form, and on $\to_2$ the action is trivial. This meas that there is a $(\p_1)_0$ equivariant map $\La^2\to_1\rightarrow \to_2$, which can be extended (uniquely) by 0 to a $(\p_1)_0$ equivariant map $\La^2\to\rightarrow \to$. However, this map is not equivariant with respect to $(\p_1)_1$, which maps $\to_2$ to $\to_1$ in a nontrivial way. Thus all the possible maximal parabolic $\h$ lack the necessary map from Lemma \ref{orbitrealthree}, and we conclude that there are no $3D$ totally real orbits.
\subsection{Orbits with a complex distribution}

\subsubsection{Subalgebras of parabolics}
In this subsection we find all subalgebras $\h$ with $\dim\h\ge 5$ of the maximal parabolics.
We consider at first all cases where\linebreak
$\h\subset \p_2\subset \sp(1,1)$. The parabolic subalgebra $\p_2$ is naturally $|1|-$graded, and can be described as $\p_2=(\su(2)\oplus \R s_2)\ltimes \R^3$, where $s_2$ is the grading element, $\R^3$ is Abelian and the action of $\su(2)$ on $\R^3=\ad(\su(2))$ is the tautological action of $\so(3)$, which is irreducible. We denote $\g_0=\su(2)\oplus \R s_2$.

Suppose that $\h$ has dimension 5 or 6. Then the intersection $\Pi=\h\cap\g_0$ is nontrivial (because of dimension) and of dimension at least 2, and $\Pi$ is a subalgebra of $\g_0$. The subalgebras of $\g_0$ of dimensions 2 and 3 are unique (up to conjugation in the former case), they are $\R t\oplus \R s_2$ and $\su(2)$, where $\R t$ is a 1D subalgebra of $\su(2)$. Note that if $\g_0\subset \h$ and $\dim \h>4$, then $\h=\p_2$ because of the irreducible action on $\R^3$. Thus up to conjugation in $\sp(1,1)$ there is one subalgebra of dimension 5, $\h=(\R t\oplus \R s_2)\ltimes \R^3$, and one subalgebra of dimension 6, $\h=\su(2)\ltimes\R^3$.

Next we consider all cases where $\h\subset \p_2\subset \sp(4,\R)$ or $\h\subset \p_1\subset \sp(4,\R)$. For both of these, we have $\g_0\simeq\sl_2(\R)\oplus \R s_i$, where $s_i\in\p_i$ is the respective grading element, but keep in mind that these subalgebras of $\p_1$ or $\p_2$ are not equivalent in $\g$, even though they are abstractly isomorphic.

We will now consider proper subalgebras of dimension $>4$ of the parabolics. Let $\Pi=\h\cap\g_0$. Similarly to above this is a subalgebra of $\g_0$ of dimension at least 2. Abstractly, the list of such is
\begin{align*}
\Pi\in\{\g_0 (*),\sl_2(\R), B_2\oplus \R s_i (*), S_2, \R k \oplus \R s_i (*) \}.
\end{align*}
Here $k\in\sl_2(\R)$, $S_2$ is a 2D solvable Lie subalgebra of $\mathfrak{gl}_2$, and we have marked with $(*)$ those subalgebras that include the grading element $s_i$. If $\Pi=\g_0$, then $\h$ is non-proper except in one case, which is $\h=\g_0\oplus\g_2\subset\p_1$, the only 5D subalgebra to have a non-trivial Levi-factor.

For the other possibilities marked with $(*)$, $\h$ must be a (possibly non-proper) subalgebra of the non-maximal parabolic $\p_{12}=\p_1\cap\p_2$, as we can take commutators with $s_i$ to produce a graded basis. If $\Pi=S_2$, then either $\h$ has a non-trivial Levi-factor, in which case $\h$ is equivalent in $\g$ to another subalgebra with $\Pi=\sl_2(\R)$, or $\h$ is solvable, in which case it is equivalent to a subalgebra of $\p_{12}$. In particular, all 5D solvable subalgebras with $\Pi=S_2$ are equivalent to subalgebras of $\p_{12}$.

The list of 6D subalgebras of $\p_2$ or $\p_1$ is thus $\sl_2(\R)\ltimes \ad(\sl_2(\R))\subset \p_2$, where $\ad(\sl_2(\R))$ is Abelian and $\sl_2(\R)$ acts on this as if it were its adjoint representation, $\sl_2(\R)\ltimes \mathfrak{heis}_3\subset \p_1$, where $\mathfrak{heis}_3$ is the 3D Heisenberg algebra and $\sl_2(\R)$ acts as derivations of $\mathfrak{heis}_3$, and the non-maximal parabolic $\p_{12}$.

In the case $\Pi=\R k \oplus \R s_i$, the algebra $\h$ is always 5D, and it depends on the conjugacy class of $\langle k \rangle$ in $\sl_2(\R)$. If $k$ has non-negative Killing norm, then $k$ is contained in a Borel subalgebra, hence also in (some conjugate of) $\p_{12}$, and so is $\h$. On the other hand, if $k$ has negative Killing norm, then it is a compact element and thus not contained in any conjugate of $B_2$ or $\p_{12}$. Thus there are two conjugacy classes of solvable 5D subalgebras which are not contained in $\p_{12}$. These have the forms $\h=(\R t \oplus \R s_2)\ltimes \ad(\sl_2(\R))$ and $\h=(\R t \oplus \R s_1)\ltimes \mathfrak{heis}_3$ for compact elements $t\in\sl_2(\R)$.

Suppose $\dim\h=5$, and $\h\subset \p_{12}\subset\sp(4,\R)$. To describe the possible subalgebras $\h$, we will use some facts about parabolic subalgebras. There are (at least) 3 possible gradings of $\p_{12}$. These are those inherited from $\p_2$ and $\p_1$, and the natural parabolic grading coming from $\p_{12}$ itself, which is different from both of the previous ones. These are respectively $|1|-$, $|2|-$ and $|3|-$ gradings. It will be most convenient for us to make use of the $|1|-$grading. This gives the description $\p_{12}=(B_2\oplus \R s_2)\ltimes \ad(\sl_2(\R))$, where $s_2$ is the grading element of $\p_2$, $B_2$ is a Borel subalgebra of $\sl_2(\R)$, and this acts on the Abelian component $\ad(\sl_2(\R))$ as if it were the restriction of the adjoint representation of $\sl_2(\R)$.

The subalgebras $\h$ are split into two cases, either $s_2\in\h$ or $s_2\not\in\h$. The former case is simpler, because if $s_2\in\h$ then we can find a basis of $\h$ where each element has pure grading. The possibilities are then $(B_2\oplus\R s_2)\ltimes \ad(B_2)$, since $B_2$ has a unique invariant subspace in $\ad(\sl_2(\R))$, or $(\R k  \oplus \R s_2)\ltimes \ad(\sl_2(\R))$ where $k$ is some element of $B_2$ (and up to equivalence there are only two examples of the latter type, with $k$ of positive and zero Killing norm).

In the case $s_2\not\in\h$, we have that $\h$ is a graph in $\p_{12}$ of some linear map\linebreak
$i:B_2\ltimes\ad(\sl_2(\R))\rightarrow \R s_2$. Any such graph defines a subspace in $\p_{12}$, but only those that are closed under the Lie bracket define subalgebras. Choosing a basis $\{e,h,f\}\subset\sl_2(\R)$ with structure relations $[e,f]=h,[h,e]=e,[h,f]=-f$, we get the $1$-graded basis $\{e_0,h_0,e_1,h_1,f_1\}\subset B_2\ltimes \ad(\sl_2(\R))$. The condition that $\h$ is a subalgebra implies that $\h=\langle e_0,h_0+\lambda s_2,e_1,h_1,f_1+\mu s_2\rangle $ where $\lambda,\mu\in\R$ are matrix entries of $i$, and the subalgebra condition is $(\lambda-1)\mu=0$. The case $\mu=0$ is $\ad(\sl_2(\R))\oplus[B_2,B_2]\subset \ker(i)$, with Lie algebra structure $\h=S_2\ltimes\ad(\sl_2(\R))$. Note that this is still $|1|-$graded. The case $\lambda=1$ corresponds to $|2|$-graded, $|1|$-ungraded algebras, as $s_1=h_0+s_2$ is the $|2|$-grading element of $\g$. The parameter $\mu$ corresponds to choosing an element $k\in\sl_2(\R)\subset \g_0$, where $\g_0\subset \p_1$, hence up to conjugation this parameter only determines whether the Killing norm of $k$ is posi\-tive or zero, and the $|2|$-graded algebra structure is $\h=(\R k\oplus \R s_1)\ltimes \mathfrak{heis}_3$.

In summary, we have the following:
\begin{prop}
Up to conjugation in $\g$, the subalgebras $\h$ of a parabolic sub\-algebra $\p\subset\g$ for $\g=\sp(1,1)$ or $\g=\sp(4,\R)$ with $\dim \h \ge 5$ are graded (in the inclusion given below) and are the following:
\end{prop}
\begin{tabular}[t]{| l | l | l | l | l |}
    \hline
    $\dim\h$    &   $\g=\sp(4,\R)$     &  $\g=\sp(1,1)$   &Grading & Notes\\  \hline
      7  &   $\p_2$    & $\p_2$ & 1 &\\  \hline
      7  &   $\p_1$    &        & 2 &\\  \hline
      6  &   $\sl_2\ltimes \ad(\sl_2(\R))$    &  $\h=\su(2)\ltimes\R^3$      & 1 &     \\  \hline
      6  &   $\sl_2\ltimes \mathfrak{heis}_3$    &       & 2 &     \\  \hline
      6  &   $\p_{12}$    &       & 1,2,3 &     \\  \hline
      5  &   $(\R t \oplus \R s_2)\ltimes \ad(\sl_2(\R))$    &  $(\R t \oplus \R s)\ltimes \R^3$      & 1 &  $||t||<0$  \\  \hline
      5  &   $(\R k  \oplus \R s_2)\ltimes \ad(\sl_2(\R))$   &         & 1 & $||k||>0$    \\  \hline
      5  &   $(\R k  \oplus \R s_2)\ltimes \ad(\sl_2(\R))$   &         & 1 & $||k||=0$   \\  \hline
      5  &   $(B_2\oplus\R s_2)\ltimes \ad(B_2)$  &         & 1 &     \\  \hline
      5  &   $S_2\ltimes\ad(\sl_2(\R))$   &         & 1 &  $\lambda\in\R$ $(\dagger)$ \\  \hline
      5  &   $\mathfrak{gl}_2(\R)\ltimes \R$   &         & 2 &   $\h=\g_0\oplus\g_2$  \\  \hline
      5  &   $(\R t \oplus \R s_1)\ltimes \mathfrak{heis}_3$    &         & 2 &  $||t||<0$  \\  \hline
      5  &   $(\R k\oplus \R s_1)\ltimes \mathfrak{heis}_3$   &         & 2 &  $||k||>0$ \\  \hline
      5  &   $(\R k\oplus \R s_1)\ltimes \mathfrak{heis}_3$   &         & 2 &  $||k||=0$ \\  \hline
\end{tabular}

The entry marked with $(\dagger)$ is a family of subalgebras which depend on a real parameter.

\subsubsection{Distributions for subalgebras of parabolics}
The orbits of each dimension inherits slightly different geometry from the complex structure and Nijenhuis tensor.

Firstly, let $\dim O = 3$, so $\dim \h=7$.
\begin{lem}
The complex distribution on a singular orbit $O$ of dimension 3, which is not totally real, has (real) dimension 2. Thus the isotropy representation of such an orbit admits a 2D invariant subspace with complex structure.
\end{lem}
\begin{proof}
It follows from the assumption that the orbit is not totally real that the distribution $L$ is non-trivial.
\end{proof}
The only 7D subalgebras are the maximal parabolics themselves. In the case of $\h=\p_2\subset \sp(1,1)$, $\g/\p_2\simeq\R^3$ with the standard action of $\su(2)\simeq \so(3)\subset \p_2$, which is irreducible. For $\h=\p_2\subset \sp(4,\R)$, we have $\sl_2(\R)\subset \p_2$, and $\g/\p_2\simeq\ad(\sl_2(\R))$, which is irreducible. For $\h=\p_1\subset \sp(4,\R)$, $\sl_2(\R)\subset \p_1$ and with respect to this $\g/\p_1\simeq\mathfrak{heis}_3\simeq\R^2\oplus \R$ where $\R^2$ has the standard $\sl_2(\R)$-action. This last submodule has the correct dimension, but even restricted to $\sl_2(\R)$ it fails to admit an invariant complex structure. Thus all of these are discarded.

Secondly, let $\dim O = 4$, so $\dim \h=6$.
\begin{lem}\label{orbitfour}
A singular orbit of dimension 4 always admits a complex distribution of real dimension 2.
\end{lem}
\begin{proof}
Let $L=\to \cap J\to$. Since $\dim O$ is even there are two cases: Either the orbit $O$ is almost complex and $\dim L =4$, or $\dim L$=2 in which case we are done. In the former case, non-degeneracy of the Nijenhuis tensor $N_J$ implies that $N_J(\La^2 L)\subset L$ is a proper and non-trivial complex submodule.
\end{proof}
There are 4 different 6D subalgebras; 3 of these have a semi-simple subalgebra, and the last is the Borel subalgebra $\p_{12}$ of $\sp(4,\R)$. Out of those with semi-simple subalgebras, two have 3D irreducible submodules in the isotropy module with respect to their Levi factor (similarly to the previous case), and only $\h=\sl_2(\R)\ltimes \mathfrak{heis}_3\subset\p_1$ has any chance of admitting a 2D submodule, as the isotropy representation decomposes as $\g/\h=\R^2\oplus\R\oplus\R$ with respect to $\sl_2(\R)$. However, the action of $\mathfrak{heis}_3$ makes the module indecomposable, and there are no submodules with respect to $\h$.

The isotropy representation of $\p_{12}$ is $|3|-$graded, $\to=\g/\p_{12}=\R^2\oplus\R\oplus\R$, and has a basis where each element spans a negative root space with respect to the Cartan subalgebra. The action of $\p_{12}$ is indecomposable (not all terms are submodules), but $\R^2$ is a 2D submodule. However, the Cartan subalgebra acts on $\R^2$ with distinct eigenvalues (roots of $\g$), and so does not preserve any complex structure.

Finally, let $\dim O = 5$, so $\dim \h=5$.
\begin{lem}
The complex distribution on a singular orbit $O$ of dimension 5 has (real) dimension 4. Thus the isotropy representation of such an orbit admits a 4D invariant subspace $L$ with complex structure. Moreover, there either exists an equivariant decomposition $\to=L\oplus\R$ where $\R$ is invariant, or there is a non-zero $\h-$invariant $L$ valued 2-form $\theta\in \La^2 L^\ast \otimes L$.
\end{lem}
\begin{proof}
The claim that $L$ has dimension at least 4 follows from the fact that this is the minimal intersection of two 5D hyperplanes ($\to$ and $J\to$) in a 6D vector space ($T_xM$). Since the distribution is complex, its dimension must be even, so it is equal to 4. By non-degeneracy of $N_J$, $N_J(\La^2 L)=\Pi$ is an invariant complex line in $T_xM$. There are two cases, either $\Pi\subset L$, in which case the map $N_J$ restricts to $L$ and gives the desired two-form (so in fact, $\theta\in \La^2 L^\ast \otimes \Pi$) , or $\Pi$ is transversal to $L$, since $\Pi$ is complex. In the latter case we may take the intersection $\to\cap \Pi$, which is an invariant line $\R\subset \to$ by dimensional count. Since $L$ and $\R$ are independent and invariant, $\to=L\oplus \R$ is the desired decomposition.
\end{proof}
Notice first that if $\h$ contains the Cartan subalgebra, then similarly to the previous case there exists a basis of root vectors in the complement to $\h$ with distinct eigenvalues (by root space decomposition), so no complex structure is possible. This rules out all the $|1|$-graded subalgebras of $\p_{12}$ except for $\h=(\R k  \oplus \R s_2)\ltimes \ad(\sl_2(\R))$, where $k$ has Killing norm $0$ in $\g$. In this case, we can find a unique 4D submodule $L$ for which the grading element $s_2$ preserves the decomposition $L=\C\oplus\C$. The grading element $s_2$ has two distinct eigenvalues, so any almost complex structure must leave these subspaces invariant, but the operator $k$ is nilpotent and commutes with $s_2$, and so acts nilpotently on the same 2D subspaces. Thus by the following lemma, the complex structures are not $\h-$invariant.
\begin{lem}
If some element $\xi\in\h$ acts as a nonzero nilpotent operator on $\Pi^2$, then $\Pi^2$ does not admit an invariant almost complex structure.
\end{lem}
\begin{proof}
We have $\End_\C(\Pi^2)=\C$, which is a field and hence does not admit nonzero nilpotent elements.
\end{proof}

Consider the parameter dependent family of subalgebras $\h=S_2\ltimes\ad(\sl_2(\R))$ of $\p_{12}$, which was marked with $(\dagger)$ in the table of subalgebras. Each member of this family admits a unique 4D submodule $L\subset\to$. The unique (up to scale) non-zero element of the Cartan subalgebra has simple spectrum when restricted to $L$. An operator which leaves a complex structure invariant must have spectrum consisting of two double (or one quadruple) values, so this does not admit any $\h-$invariant complex structure.

For the $|2|-$graded $\h=(\R s_1\oplus \R k)\ltimes \mathfrak{heis}_3$, there is a unique 4D submodule $L$, and the spectrum of the $|2|-$grading element $s_1$ consists of two double values when restricted to $L$. However $k$ commutes with $s_1$, and either acts as a non-zero nilpotent operator or with simple spectrum on $L$, depending on its Killing norm. In either case, $L$ does not admit an invariant complex structure.

The last parabolic cases are those that contain an element $t$ with negative Killing norm. There are three such 5D subalgebras  These have the forms $\h=(\R t \oplus \R s_2)\ltimes \ad(\sl_2(\R))$ and $\h=(\R t \oplus \R s_1)\ltimes \mathfrak{heis}_3$. The former case is contained in $\p_2$ of both $\sp(4,\R)$ and $\sp(1,1)$ with identical isotropy representations, while the latter is contained in $\p_1$ of $\sp(4,R)$. In all cases, the isotropy representation decomposes as $\C\oplus\C\oplus \R$ with respect to $\R t$, and $L=\C\oplus\C$ is invariant with respect to $\h$. Note that the $\R$ term is transversal to $L$, but it is not invariant under $\h$, and neither is any other transversal. Therefore, by Lemma \ref{orbitfour}, there must be an invariant non-zero vector valued two-form on $L$ if $\h$ is the isotropy of a singular orbit. However we compute $(\La^2(L^*)\otimes L)^\h={0}$. Thus this case cannot appear as a singular orbit, in spite of being the only case to admit the required complex structure on the 4D distribution.

To conclude: No subalgebra $\h$ of a parabolic $\p_i$ with $\dim\h>4$ can occur as the isotropy of a singular orbit.

\subsubsection{Subalgebras of maximal semi-simple}
The (complex) rank of the complexification of $\g$ is 2, hence the complexification of a maximal semi-simple subalgebra $\h$ can have rank at most 2 as well. We must also have $\dim\h>4$, so if $\h$ is proper and maximal, then it has $\dim\h=6$ and $\h$ is a real form of $A_1\oplus A_1$. These are $\mathfrak{so}(2,2)\simeq\mathfrak{sl}(2)\oplus \mathfrak{sl}(2)$, $\mathfrak{so}(4)\simeq\mathfrak{su}(2)\oplus \mathfrak{su}(2)$, $\mathfrak{so}(1,3)\simeq\mathfrak{sl}_2(\C)_\R$ and $\mathfrak{sl}(2)\oplus \mathfrak{su}(2)$. A 5D subalgebra of one of these is also possible.

Out of these, $\mathfrak{sl}(2)\oplus \mathfrak{su}(2)$ does not admit any 5D real faithful representation with an invariant metric of any signature, so this does not embed into any real form of $B_2$. The other forms embed into $\mathfrak{so}(5)$, $\mathfrak{so}(1,4)$ or $\mathfrak{so}(2,3)$ according to the signature of their invariant metrics on the defining representation $\R^5=\R^4\oplus\R$, where the last $\R$ term is trivial and $\R^4$ is tautological. (There are also other embeddings, but this will cover all the correct pairs of algebra/subalgebra.)

In all these cases, the isotropy representation is a faithful 4D real representation. While this is enough to conclude that there is no invariant complex structure for $\mathfrak{so}(4)$ and $\mathfrak{so}(2,2)$, $\mathfrak{so}(1,3)$ does in fact have a 4D rep with complex structure: the standard action of $\mathfrak{sl}_2(\C)$ on $\C^2$, but since this is real irreducible it cannot admit a nonzero Nijenhuis tensor. Thus all maximal semi-simple subalgebras are excluded.

Out of these real forms, the only one which admits a 5D subalgebra is $\mathfrak{so}(2,2)$, which has the subalgebra $\p_1=B_2\oplus \sl_2(\R)$, where $B_2$ is the Borel subalgebra of the other copy of $\sl_2(\R)$. This algebra yields an invariant splitting $\to=\R\oplus \R^4$. The action on $\R^4$ by $\p_1$ is the one which comes from the embedding to $\mathfrak{so}(2,2)$, and this does not admit any invariant complex structure.

The proof of Theorem 2 is now complete.


\end{document}